\documentclass{article}

\usepackage{a4wide}
\usepackage{amsmath}
\usepackage{amsthm}
\usepackage{amssymb}

\usepackage[colorlinks=true,urlcolor=blue,linkcolor=blue,plainpages=false,pdfpagelabels]{hyperref}

\newtheorem{theorem}{Theorem}[section]

\newtheorem{proposition}[theorem]{Proposition}
\newtheorem{corollary}[theorem]{Corollary}
\newtheorem{lemma}[theorem]{Lemma}

\newcommand{\N}{{\mathbb N}}
\newcommand{\abs}[1]{\left\vert#1\right\vert}
\newcommand{\ceil}[1]{\left\lceil#1\right\rceil}

\title{Linear rates of asymptotic regularity for Halpern-type iterations}
\author{Hora\c tiu Cheval${}^a$ and Lauren\c{t}iu Leu\c{s}tean${}^{a,b,c}$\\[2mm]
\footnotesize ${}^a$LOS, Faculty of Mathematics and Computer Science, University of Bucharest\\
\footnotesize ${}^b$Simion Stoilow Institute of Mathematics of the Romanian Academy\\
\footnotesize ${}^c$Institute for Logic and Data Science, Bucharest\\
\footnotesize E-mails:  horatiu.cheval@unibuc.ro, laurentiu.leustean@unibuc.ro
}

\date{}
\begin{document}

\maketitle

\begin{abstract}
In this note we apply a lemma due to Sabach and Shtern to compute linear 
rates of asymptotic regularity for Halpern-type nonlinear iterations studied in optimization and nonlinear analysis.\\



\end{abstract}

\section{Introduction}

Sabach and Shtern proved in \cite{SabSht17} a lemma that can be used to show the linear convergence towards $0$ 
of sequences of nonnegative reals satisfying certain conditions. 
They used this lemma to get linear rates of asymptotic regularity for the sequential 
averaging method (SAM),  an iteration defined by Xu \cite{Xu04} as a viscosity approximation 
of the Halpern iteration. As an immediate consequence, linear rates for the Halpern iteration 
are obtained. The lemma of Sabach and Shtern has since been employed to obtain linear rates of asymptotic 
regularity for the Tikhonov-Mann and modified Halpern iterations in \cite{CheKohLeu23} and 
in \cite{LeuPin23} for the alternating Halpern-Mann iteration, introduced recently by Dinis and Pinto \cite{DinPin23}.

The following lemma is a slight reformulation of \cite[Lemma~3]{SabSht17}, proved in \cite[Lemma~2.8]{LeuPin23}.

\begin{lemma}\label{lem:sabach-shtern-v2}
Let $L > 0$, $J \geq N \geq 2$, $\gamma \in (0, 1]$, $(c_n)$ be a sequence bounded above by $L$, and 
$a_n = \frac{N}{\gamma(n + J)}$ for all $n\in\N$. Suppose that $(s_n)$ is a sequence of nonnegative 
reals such that $s_0 \leq L$ and, for all $n \in \N$,
\begin{align}
s_{n + 1} \leq (1 - \gamma a_{n + 1}) s_n + (a_n - a_{n + 1}) c_n. \label{eq:sabach-shtern-main-ineq}
\end{align}
Then, for all $n \in \N$,
\begin{align}
s_n \leq \frac{J L}{\gamma(n + J)}.
\end{align}
\end{lemma}

In this paper we show that Lemma~\ref{lem:sabach-shtern-v2} can be applied to compute linear 
rates of asymptotic regularity for other Halpern-type nonlinear iterations studied in optimization and nonlinear analysis.

\section{Preliminaries}

A $W$-space is a structure $(X,d,W)$, where $(X,d)$  is a metric space and 
$W:X\times X\times [0,1]\to X$. We think of $W(x,y,\lambda)$ as an abstract convex 
combination of the points $x,y\in X$. That is why we  write $(1-\lambda)x + \lambda y$ 
instead of $W(x,y,\lambda)$. 

A $W$-hyperbolic space \cite{Koh05} is a $W$-space satisfying for all $x,y,w,z\in X$, $\lambda,\tilde{\lambda}\in[0,1]$,
$$
\begin{array}{ll}
\text{(W1)} & d(z,(1-\lambda)x + \lambda y) \le  (1-\lambda)d(z,x)+\lambda d(z,y),\\[1mm]
\text{(W2)} & d((1-\lambda)x + \lambda y,(1-\tilde{\lambda})x + \tilde{\lambda} y) =  \vert \lambda-\tilde{\lambda} \vert d(x,y), \\[1mm]
\text{(W3)} & (1-\lambda)x + \lambda y = \lambda y + (1-\lambda)x, \\[1mm]
\text{(W4)} & d((1-\lambda)x + \lambda z,(1-\lambda)y + \lambda w) \le (1-\lambda)d(x,y)+\lambda d(z,w).
\end{array}
$$
It is obvious that (convex subsets of) normed spaces are $W$-hyperbolic spaces. Other examples of $W$-hyperbolic spaces
are  Busemann spaces (as proved in \cite[Proposition~2.6]{AriLeuLop14}) and 
CAT(0) spaces (as proved in \cite[pp. 386-388]{Koh08}). 

We usually write $X$ instead of $(X,d,W)$. A nonempty subset $C$ of $X$ is said to be convex if 
for all $x,y\in C$, we have that $(1-\lambda)x + \lambda y \in C$ for every $\lambda\in[0,1]$.  

\begin{proposition}
Let $X$ be a $W$-hyperbolic space. For all $x,y,w,z\in X$, $\lambda, \widetilde{\lambda}\in[0,1]$, 
\begin{align}
d(x,(1-\lambda)x + \lambda y)=\lambda d(x,y) \text{~and~} d(y,(1-\lambda)x + \lambda y)=(1-\lambda)d(x,y), 
\label{dxlambday}\\
d((1 - \lambda) x + \lambda z, (1 - \widetilde{\lambda}) y + \widetilde{\lambda} w) \leq (1 - \lambda) d(x, y) + \lambda d(z, w) + 
\abs{\lambda - \widetilde{\lambda}} d(y, w). \label{prop:arbitrary-w-ineq}
\end{align}
\end{proposition}
\begin{proof}
\eqref{dxlambday} holds even in the more general setting of convex metric spaces, introduced 
by Takahashi \cite{Tak70} as $W$-spaces satisfying (W1). For the proof of \eqref{prop:arbitrary-w-ineq} see
\cite[Lemma~2.1(iv)]{CheLeu22}.
\end{proof}

Asymptotic regularity is a very important notion in nonlinear analysis and optimization that was defined by 
Browder and Petryshyn \cite{BroPet66} for the Picard iteration of a nonexpansive mapping and extended, subsequently, by 
Borwein, Reich and Shafrir \cite{BorReiSha92} to general nonlinear iterations. It turns out that 
two notions of asymptotic regularity for a sequence $(a_n)$ in a metric space $(X,d)$ arise naturally. 
Firstly, we say that $(a_n)$ is asymptotically regular if $\lim\limits_{n\to \infty} d(a_n,a_{n+1})=0$. 
If, furthermore, $C$ is a nonempty subset of $X$ and $T:C\to C$ is a mapping, then $(a_n)$ is said to be 
$T$-asymptotically regular if $\lim\limits_{n\to \infty} d(a_n,Ta_n)=0$. 

If $(a_n)$ is asymptotically regular, then a rate of convergence of  $(d(a_n,a_{n+1}))$ towards $0$  
is called a rate of asymptotic regularity of $(a_n)$. Similarly, if $(a_n)$ is 
$T$-asymptotically regular, then a rate of convergence of $(d(a_n,Ta_n ))$ towards $0$ is 
called a rate of $T$-asymptotic regularity of $(a_n)$.

As in \cite{LeuPin21}, we can extend the notion of $T$-asymptotic regularity to countable families of mappings. Thus,
if $(T_n : C \to C)$ is a  countable family, then we say that $(x_n)$ is $(T_n)$-asymptotically regular with rate $\varphi$ if 
$\lim\limits_{n\to \infty}d(x_n, T_nx_n)=0$ with rate of convergence $\varphi$.

Recall that if $(b_n)$ is a sequence of nonnegative reals such that $\lim\limits_{n\to \infty}b_n=0$, then 
a rate of convergence of $(b_n)$ (towards $0$) is a function $\varphi:\N\to \N$ satisfying 
\[\forall k\in\N\, \forall n\geq \varphi(k) \left(b_n\leq \frac1{k+1}\right).\] 
We denote by $\N^*$ the set $\N\setminus \{0\}$ of positive natural numbers.

\section{Linear rates of ($T$-)asymptotic regularity}

In the sequel, if not otherwise stated, $X$ is a $W$-hyperbolic space, 
$C$ is a convex subset of $X$, and $T:C\to C$ is a nonexpansive mapping, that is 
$d(Tx,Ty)\leq d(x,y)$ for all $x,y\in C$. We denote by $\operatorname{Fix}(T)$ the set of fixed points of $T$ and we assume that $T$ has fixed points, 
i.e. $\operatorname{Fix}(T)\ne\emptyset$.

\subsection{Halpern iteration}

Let us recall that the well-known Halpern iteration is defined as follows: 
\begin{equation}
x_0=x, \quad x_{n + 1}  = (1 - \alpha_n) u + \alpha_n Tx_n, \label{def-Halpern}
\end{equation}
where $x,u\in C$ and $(\alpha_n)$ is a sequence in $[0,1]$.

In \cite{CheKohLeu23}, the authors and Kohlenbach computed linear rates of ($T-$)asymptotic regularity, 
for a special choice of the parameter sequences,
for the modified Halpern iteration, a generalization of the Halpern iteration. 
One gets immediately linear rates for this iteration too. 

\begin{proposition}\label{Halpern-rates}
Let $\alpha_0=0$ and $\alpha_n=1-\frac{2}{n+1}$ for $n\ge 1$. Define
\begin{equation}\label{def-rates-Halpern}
\Phi(k) = 8M(k+1)-1 \quad \text{and} \quad \Phi^*(k)  = 16M(k+1)-1,
\end{equation}
where $M\in\N^*$ is such that $M\ge \max\{d(x,p), d(u,p)\}$ for some fixed point $p$ of $T$. 

Then $\Phi$ is a rate of asymptotic regularity for $(x_n)$ and $\Phi^*$  is a rate of $T$-asymptotic 
regularity for $(x_n)$.
\end{proposition}
\begin{proof}
Apply \cite[Proposition~4.7]{CheKohLeu23} with $\lambda_n=1$ and $\beta_{n+1}=\alpha_n$. 
\end{proof}

Remark that if $C$ is a bounded set, then one can take $M\in \N^*$ to be an upper bound on the diameter of $C$.\\

For Hilbert spaces, $u=x$ and $\alpha_n=1-\frac1{n+2}$, Lieder \cite{Lie21} proved the 
following inequality: for all $n\geq 1$ and all fixed points $p$ of $T$, 
\begin{equation}\label{ineq-Lieder}
\|x_n-Tx_n\|\leq \frac{2\|x_0-p\|}{n+1}.
\end{equation}
Furthermore, Lieder showed that the bound in \eqref{ineq-Lieder} is tight.  As a consequence, one 
gets that $\Phi_0(k)= 2M(k+1)-1$, where $M\in \N^*$ is such that $M\geq \|x_0-p\|$, is a 
rate of $T$-asymptotic regularity that improves by a factor of $8$ the one for $W$-hyperbolic spaces. \\

In a recent paper \cite{HeXuDonMei24},  He et al. defined  the following Halpern iteration with adaptive anchoring parameters:
\begin{align}\label{def-halpern-adaptive}
  x_0 \in X, \quad x_{n + 1} = \frac{1}{\varphi_{n + 1} + 1} x_0 + \frac{\varphi_{n + 1}}{\varphi_{n + 1} + 1}T x_n, 
\end{align}
where $X$ is a Hilbert space, $T:X\to X$ is nonexpansive with $\operatorname{Fix}(T)\ne\emptyset$, and, for every $n\geq 1$, 
$$\varphi_{n} = \dfrac{2\langle x_{n - 1} - T x_{n - 1}, x_0 - x_{n - 1}\rangle}{\|x_{n - 1} - T x_{n - 1}\|^2} + 1.$$
They proved (see \cite[Lemma~3.1.(i)]{HeXuDonMei24} and  \cite[Theorem~3.2]{HeXuDonMei24})  that  for  $n\geq 1$ and  any fixed point $p$ of $T$,
\begin{equation}\label{rates-T-as-reg-halpern-adaptive}
\varphi_n \geq n \quad \text{and} \quad \|x_n - T x_n\| \leq \frac{2\|x_0- p\|}{\varphi_n + 1},
\end{equation}
that is in general an impovement of  \eqref{ineq-Lieder}. Furthermore, \cite[Remark~3.3]{HeXuDonMei24} provides an example  showing that the rate of $T$-asymptotic regularity given by \eqref{rates-T-as-reg-halpern-adaptive} is better than the one from \cite{Lie21}. 

As the definition of $\varphi_n$ uses essentially the Hilbert space structure, the proof of \eqref{rates-T-as-reg-halpern-adaptive}  from \cite{HeXuDonMei24} cannot be generalized to Banach spaces or to  $W$-hyperbolic spaces. Furthermore, the proof does not use Sabach and Shtern's lemma. A problem for future research is to  find  an adaptive selection method for $\varphi_n$ that can be  expressed in  a more general setting (Banach spaces, classes of $W$-spaces) and to compute rates of $T$-asymptotic regularity for the corresponding adaptive Halpern iteration having a similar form with the ones from \cite{HeXuDonMei24}.

\subsection{Sequential averaging method}

Xu \cite{Xu04} introduced, in the setting of Banach spaces, a viscosity version of the Halpern 
iteration that was called the  sequential averaging method (SAM) by Sabach and Shtern \cite{SabSht17}, 
who studied it in  Hilbert spaces. Sabach and Shtern obtained, as an application of \cite[Lemma~3]{SabSht17}, 
linear rates of asymptotic regularity for SAM. In what follows, we show that we can obtain such linear rates in the much more general setting 
of $W$-hyperbolic spaces, too.

Let $\rho \in [0,1)$ and $f: C\to C$ be a $\rho$-contraction, that is for all 
$x,y\in C$, $d(f(x), f(y))\leq \rho d(x,y)$. 
The sequential averaging method (SAM)  is defined by
\begin{equation}
x_0=x, \qquad x_{n + 1}  = (1 - \alpha_n) f(x_n) + \alpha_n Tx_n, \label{def-SAM}
\end{equation}
where $x\in C$ and $(\alpha_n)$ is a sequence in $[0,1]$.

The following lemma extends to $W$-hyperbolic spaces results obtained for Banach spaces in \cite[proof of Theorem~3.2]{Xu04}. 

\begin{lemma}\label{lem:SAM}
\begin{enumerate}
\item 
For all $n \in \N$, 
\begin{align}
d(x_{n + 2}, x_{n + 1}) & \leq (1-(1-\rho)(1-\alpha_{n+1})) d(x_{n + 1}, x_n)+ |\alpha_{n+1} - \alpha_n|d(f(x_n),Tx_n), 
\label{main-ineq-SAM}\\
d(x_n, T x_n) & \leq d(x_n,x_{n + 1}) + (1-\alpha_n) d(f(x_n), Tx_n). \label{main-ineq-SAM-2}
\end{align}
\item\label{lem:SAM-ii} 
For all $n \in \N$ and any fixed point $p$ of $T$, 
\begin{align}
d(x_n, p) &\leq M,\label{xnp-SAM}\\
d(f(x_n), Tx_n) &\leq 2M, \label{fxn-Txn-SAM}
\end{align}
where $M\in \N^*$ is such that $M\geq \max\left\{d(x, p), \frac{d(p,f(p))}{1 - \rho}\right\}$. 
\end{enumerate}
\end{lemma}
\begin{proof} 
\begin{enumerate}
\item Let $n\in\N$.  We have, by \eqref{prop:arbitrary-w-ineq}, that 
\begin{align*}
d(x_{n + 2}, x_{n + 1}) 
&\leq (1 - \alpha_{n+1}) d(f(x_{n + 1}), f(x_n)) + \alpha_{n+1} d(T x_{n + 1}, T x_n) \\
& \quad + |\alpha_{n+1} - \alpha_n| d(f(x_n), T x_n).
\end{align*}
As $f$ is a $\rho$-contraction and $T$ is nonexpansive, we get \eqref{main-ineq-SAM}.
Moreover, 
\begin{align*}
d(x_n,Tx_n) &\leq d(x_n,x_{n + 1})+ d(x_{n + 1},T x_n) \stackrel{\eqref{dxlambday}}{=} d(x_n,x_{n + 1})+ 
(1-\alpha_n) d(f(x_n), Tx_n).
\end{align*}
\item We prove \eqref{xnp-SAM} by induction on $n$. The case $n=0$ is obvious.
$n \Rightarrow n+1$: 
\begin{align*}
d(x_{n + 1}, p) 
& \leq (1 - \alpha_n)d(f(x_n),p) + \alpha_n d(Tx_n, p) \quad \text{by (W1)}\\ 
&\leq  (1 - \alpha_n)d(f(x_n),f(p)) + (1 - \alpha_n)d(f(p),p) +\alpha_n d(x_n, p) \\
&\leq (1 - \alpha_n)\rho d(x_n,p)+\alpha_n d(x_n, p) + (1 - \alpha_n)d(f(p),p)\\
&\leq (1 - \alpha_n)\rho M +\alpha_n M + (1 - \alpha_n)(1-\rho)M =M,
\end{align*}
by the definition of $M$ and the induction hypothesis. Furthermore, for all $n\in\N$, 
\begin{align*}
d(f(x_n), Tx_n) 
&\leq d(f(x_n),f(p))+ d(f(p),p)+ d(p, T x_n)  \leq \rho d(x_n,p)+ d(f(p),p) + d(x_n,p)\\
& \stackrel{\eqref{xnp-SAM}}{\leq} (1+\rho)M + d(f(p),p) \leq 2M.
\end{align*}
\end{enumerate}
\end{proof}

\begin{proposition}\label{main-prop-SAM}
Let $\alpha_n = 1-\frac{2}{(1 - \rho)(n + J)}$ for all $n\in\N$, where $J = 2\ceil{\frac{1}{1 - \rho}}$. 
Then for all $n \in \N$,
\begin{align}
d(x_{n + 1}, x_n) & \leq \frac{2MJ}{(1 - \rho)(n + J)}, \label{SAM-as-reg}\\
d(T x_n, x_n) & \leq \frac{2M(J + 2)}{(1 - \rho)(n + J)}, \label{SAM-T-as-reg}
\end{align}
where $M$ is as in Lemma~\ref{lem:SAM}.  
\end{proposition}
\begin{proof}
Define 
\begin{align*}
L=2M, \quad N=2, \quad \gamma = 1 - \rho, \quad s_n = d(x_{n + 1}, x_n), 
\quad a_n = 1-\alpha_n, \quad  c_n = d(f(x_n), Tx_n).
\end{align*}
One can easily see that the assumptions of  Lemma~\ref{lem:sabach-shtern-v2} are met, so we can apply it 
to get \eqref{SAM-as-reg}.

\eqref{SAM-T-as-reg} follows easily:
\begin{align*}
d(x_n, T x_n) & \leq d(x_n,x_{n + 1}) + 2M(1-\alpha_n) \quad \text{by~} \eqref{main-ineq-SAM-2} \text{~and~}\eqref{fxn-Txn-SAM}\\
& \leq \frac{2M(J+2)}{(1 - \rho)(n + J)},
\end{align*}
by \eqref{SAM-as-reg} and the definition of $(\alpha_n)$.
\end{proof}

As an immediate consequence of Proposition~\ref{main-prop-SAM}, we get, for $\alpha_n = 1-\frac{2}{(1 - \rho)(n + J)}$,
linear rates of ($T$-)asymptotic regularity. Thus, 
$(x_n)$ is asymptotically regular with rate 
\begin{equation}\label{def-rate-as-reg-SAM}
\Phi(k)= 4M\ceil{\frac{1}{1 - \rho}}^2(k+1)-2\ceil{\frac{1}{1 - \rho}}
\end{equation}
and $T$-asymptotically regular with rate 
\begin{equation}\label{def-rate-T-as-reg-SAM}
\Phi^*(k)= \left(4M\ceil{\frac{1}{1 - \rho}}^2+4M\ceil{\frac{1}{1 - \rho}}\right)(k+1)-2\ceil{\frac{1}{1 - \rho}}.
\end{equation}

Let $u\in C$ and take $f:C\to C, \, f(x)=u$. Then $\rho=0$, $J=2$, $M\geq \max\{d(x_0, p), d(u,p)\}$. 
Furthermore, $(x_n)$ becomes the Halpern iteration.  We get the following:

\begin{corollary}
Let $\alpha_n = 1-\frac{2}{n + 2}$ for all $n\in\N$, and define
\begin{equation}
\Phi_0(k)=4M(k+1)-2, \qquad \Phi^*_0(k)=8M(k+1)-2.
\end{equation}
Then $\Phi_0$ is  
a rate of asymptotic regularity  and  $\Phi^*_0$ is a rate of $T$-asymptotic regularity for the Halpern iteration. 
\end{corollary}
These rates are better by a factor of $2$ than the ones obtained in Proposition~\ref{Halpern-rates}. \\

Very recently \cite{Zas24}, Zaslavski extended the above results  by weakening  the hypothesis that $T$ has fixed points to $T$ having approximate fixed points in a bounded neighborhood of $x$, as guaranteed by results in proof mining (see, for example \cite{GerKoh08,Koh08}), or that $T$ has bounded orbits.

\subsection{Alternative iterative method}

Let $(\alpha_n)$ be a sequence in $[0, 1]$ and $f: C\to C$ be  a $\rho$-contraction for some $\rho \in [0,1)$.
 The alternative iterative method $(x_n)$ is defined as follows:
\begin{equation}
 x_0=x\in C, \qquad    x_{n + 1} = T((1 - \alpha_n)f(x_n)+\alpha_n x_n). \label{def-aim}
\end{equation}
This iteration was studied in \cite{ColLeuLopMar11} in the setting of Banach spaces. 
By letting $f(x)=u\in C$ in \eqref{def-aim}, one obtains, as a particular case, the alternative regularization method, introduced by Xu \cite{Xu10}. 
The second author and Nicolae remarked in \cite{LeuNic17} that results on the asymptotic 
regularity and the strong convergence of the alternative regularization method can be immediately 
obtained from the ones for the Halpern iteration. 

The next lemma extends in a straightforward way \cite[Lemma~2.5]{ColLeuLopMar11} from normed spaces to  $W$-hyperbolic 
spaces and is very similar to Lemma~\ref{lem:SAM}. 

\begin{lemma}\label{aim-main-lemma}
\begin{enumerate}
\item For all $n \in \N$, 
\begin{align}
d(x_{n + 2}, x_{n + 1}) & \leq (1-(1-\rho)(1-\alpha_{n+1}))d(x_{n + 1}, x_n) + \abs{\alpha_n - \alpha_{n + 1}} d(x_n, f(x_n)),
\label{aim-main-ineq}\\
d(x_n, T x_n) & \leq d(x_{n + 1},x_n) + (1 - \alpha_n)d(x_n, f(x_n)). \label{aim-main-T-ineq} 
\end{align}
\item\label{aim-dxnp} For all $n \in \N$ and any fixed point $p$ of $T$, 
\begin{center}
$d(x_n, p) \leq M$ and $d(x_n, f(x_n))\leq 2M$,
\end{center}
with $M$ defined as in Lemma~\ref{lem:SAM}.\eqref{lem:SAM-ii}.
\end{enumerate}
\end{lemma}

One can easily see, by using the previous lemma, that Proposition~\ref{main-prop-SAM} holds for the alternative iterative method, too.
Furthermore, one gets for the alternative iterative method exactly the same rates of $(T-)$asymptotic regularity as for SAM.

\subsection{Halpern-type abstract proximal point algorithm}

Let $(T_n : X \to X)$ be a family of self-mappings of $X$ that  have common fixed points. 
Hence, if we denote by $F$ the set $\bigcap_{n\in\N} \operatorname{Fix}(T_n)$, then $F\ne \emptyset$.

The \emph{Halpern-type abstract proximal point algorithm} was defined by Sipo\c{s} \cite{Sip22} as follows:
\begin{align}
x_0=x, \qquad & x_{n + 1} = (1 - \alpha_n)u + \alpha_n T_n x_n, \label{eq:happa}
\end{align}
where $x,u\in C$ and $(\alpha_n)$ is a sequence in $[0, 1]$.\\

Throughout this section, $p\in F$ and $M\in\N^*$ is such that $M\geq \max\left\{d(x, p), d(u, p)\right\}$. 

\begin{lemma}\label{prop:happa-bounds} 
For all $n \in \N$, $d(x_n, p) \leq M$.
\end{lemma}
\begin{proof} An easy induction on $n$.
\end{proof}

From now on we assume that $(\gamma_n)$ is a sequence of positive reals such that for all $y\in X$ and $n, m \in \N$,
\begin{align}\label{main-ineq-HPPA}
d(T_n y, T_m y) \leq \frac{\abs{\gamma_n - \gamma_m}}{\gamma_n} d(y, T_n y). 
\end{align}
This is condition (C1) from  \cite{LeuNicSip18}, used to obtain abstract versions of the proximal point algorithm. 
As proved in \cite[Proposition~3.10]{LeuNicSip18}, \eqref{main-ineq-HPPA} holds if $X$ is a CAT(0) space and the family $(T_n)$ is 
jointly $(P_2)$ with respect to $(\gamma_n)$, a notion introduced in \cite{LeuNicSip18} as a generalization of 
the so-called $(P_2)$ property \cite{AriLopNic15}. Families of mappings $(T_n)$ that are jointly 
firmly nonexpansive with respect to $(\gamma_n)$ are also defined and studied in \cite{LeuNicSip18} as an extension of the definition of 
a firmly nonexpansive mapping in $W$-hyperbolic spaces, given in \cite{AriLeuLop14}. 
Jointly firmly nonexpansive families of mappings are jointly $(P_2)$ with respect to $(\gamma_n)$, hence they also satisfy \eqref{main-ineq-HPPA} in CAT(0) spaces.
We refer to \cite{LeuNicSip18,Sip22a} for more information about these families of mappings.

The following result establishes the essential inequality that allows us to apply Lemma~\ref{lem:sabach-shtern-v2}.

\begin{lemma}
For all $n \in \N$, 
\begin{align}
d(x_{n + 2}, x_{n + 1}) & \leq \alpha_{n + 1} d(x_{n + 1}, x_n) +\frac{2M\alpha_{n + 1}\abs{\gamma_n - \gamma_{n + 1}}}{\gamma_n}
+2M \abs{\alpha_{n + 1} - \alpha_n}. \label{eq:happa-main-ineq}
\end{align}
\end{lemma}
\begin{proof}
We have that for all $n\in\N$, 
\begin{align*}
d(x_{n + 2}, x_{n + 1}) 
&= d((1 - \alpha_{n + 1}) u + \alpha_{n + 1} T_{n + 1} x_{n + 1}, (1 - \alpha_n) u + \alpha_n T_n x_n) \\ 
& \leq \alpha_{n + 1} d(T_{n + 1} x_{n + 1}, T_n x_n) 
 + \abs{\alpha_{n + 1} - \alpha_n} d(u, T_n x_n) \quad\text{by \eqref{prop:arbitrary-w-ineq}}\\ 
&\leq \alpha_{n + 1} d(T_{n + 1} x_{n + 1}, T_{n + 1} x_n) + \alpha_{n + 1} d(T_{n + 1} x_n, T_n x_n) + 
\abs{\alpha_{n + 1} - \alpha_n} d(u, T_n x_n) \\ 
& \leq \alpha_{n + 1} d(x_{n + 1}, x_n) + \frac{\alpha_{n + 1}\abs{\gamma_n - \gamma_{n + 1}}}{\gamma_n} d(x_n, T_n x_n) + 
\abs{\alpha_{n + 1} - \alpha_n} d(u, T_n x_n), 
\end{align*}
by \eqref{main-ineq-HPPA} and the nonexpansiveness of $T_{n + 1}$. Apply Lemma~\ref{prop:happa-bounds}
to obtain that 
$d(x_n, T_n x_n) \leq d(T_n x_n, p) + d(x_n, p) \leq  2 d(x_n,p)\leq 2M$ and 
$d(u, T_n x_n) \leq d(T_n x_n, p) + d(p, u)\leq d(x_n,p) + d(p,u) \leq 2M$. 
Then \eqref{eq:happa-main-ineq} follows. 
\end{proof}

\begin{lemma}\label{prop:T-as-reg-ineq}
For all $m, n \in \N$,
\begin{align}
d(x_n, T_n x_n) & \leq d(x_n, x_{n + 1}) + 2M(1 - \alpha_n), \label{Tn-as-reg-ineq}\\
d(x_n, T_m x_n) & \leq d(x_n, T_nx_n)+\frac{\abs{\gamma_n - \gamma_m}}{\gamma_n}d(x_n, T_nx_n). \label{Tm-as-reg-ineq}
\end{align}
\end{lemma}
\begin{proof}
Let $n \in \N$. Then
\begin{align*}
d(x_n, T_n x_n) 
&\leq d(x_n, x_{n + 1}) + d(x_{n + 1}, T_n x_n) \stackrel{\eqref{dxlambday}}{=} d(x_n, x_{n + 1}) + (1 - \alpha_n) d(u, T_n x_n) \\
&\leq  d(x_n, x_{n + 1}) + 2M(1 - \alpha_n).
\end{align*}
and $d(x_n, T_m x_n) \leq d(x_n, T_nx_n) +d(T_nx_n, T_mx_n) \stackrel{\eqref{main-ineq-HPPA}}{\leq} 
d(x_n, T_nx_n)+\frac{\abs{\gamma_n - \gamma_m}}{\gamma_n}d(x_n, T_nx_n)$.
\end{proof}

\begin{proposition}\label{happa-main}
Let $\alpha_n = 1 - \frac{2}{n + 2}$ and $\gamma_n=\frac{n+2}{n+1}$ for all $n\in\N$. 
The following hold for all $m,n\in \N$:
\begin{align*}
d(x_{n + 1}, x_n)  \leq \frac{6M}{n + 2}, \quad 
d(x_n, T_n x_n) \leq \frac{10M}{n + 2},  \quad \text{and~}
d(x_n, T_m x_n)  \leq \frac{20M}{n + 2}.  
\end{align*}
\end{proposition}
\begin{proof}
We use Lemma~\ref{lem:sabach-shtern-v2} with $L=3M$, $N=J=2$,  
$\gamma=1$, $s_n = d(x_{n + 1}, x_n)$, $a_n = 1 - \alpha_n=\frac2{n+2}$, and $c_n = \frac{2M\alpha_{n + 1}(\gamma_n - \gamma_{n + 1})}{(\alpha_{n + 1}-\alpha_n)\gamma_n}+2M$.  As $\alpha_{n + 1}-\alpha_n=\frac2{(n+2)(n+3)}>0$ and $\gamma_n - \gamma_{n+1}=\frac{1}{(n+1)(n+2)}>0$, we get, by   \eqref{eq:happa-main-ineq}, that for all $n\in\N$, 
\begin{align*}
s_{n+1} &= d(x_{n + 2}, x_{n + 1}) \leq \alpha_{n + 1} d(x_{n + 1}, x_n) +(\alpha_{n + 1} - \alpha_n)
\left(\frac{2M\alpha_{n + 1}(\gamma_n - \gamma_{n + 1})}{\gamma_n(\alpha_{n + 1} - \alpha_n)} + 2M\right)\\
& = (1 - \gamma a_{n + 1}) s_n + (a_n - a_{n + 1}) c_n.
\end{align*}
Furthermore, $c_n= \frac{M(n+1)}{n+2}+2M<3M=L$, and $s_0=d(x_1,x_0)\leq d(x_1,p)+d(x_0,p) \leq 2M<L$, by Lemma~\ref{prop:happa-bounds}. Hence, we can apply Lemma~\ref{lem:sabach-shtern-v2} to get the first inequality. 
The second inequality follows by \eqref{Tn-as-reg-ineq}. Furthermore, the last inequality holds by 
\eqref{Tm-as-reg-ineq} and the fact that  for all $m,n\in\N$, $ \frac{\abs{\gamma_n - \gamma_m}}{\gamma_n} = \frac{\abs{m-n}}{(m+1)(n+2)}<1$.
\end{proof}

Proposition~\ref{happa-main} gives us linear rates for the Halpern-type 
abstract proximal point algorithm.

\begin{proposition}\label{linear-rates-HAPP}
Let $\alpha_n = 1 - \frac{2}{n + 2}$ and $\gamma_n=\frac{n+2}{n+1}$ for all $n\in\N$.  Define
\begin{align}\label{linear-rates-HAPPA}
\Phi(k) = 6M(k+1)-2, \quad \Phi^*(k)= 10M(k+1)-2, \quad and \quad \tilde{\Phi}(k)= 20M(k+1)-2. 
\end{align}
Then $(x_n)$ is asymptotically regular with rate $\Phi$, $(T_n)$-asymptotically regular  with rate $\Phi^*$, 
and, for all $m\in\N$, $T_m$-asymptotically regular with rate $\tilde{\Phi}$. 
\end{proposition}

\subsection{Halpern-type proximal point algorithm}

Let $X$ be a Hilbert space and $A : X \to 2^X$ be a maximally monotone operator such that the set 
$\operatorname{zer}(A) = \{x \in X \mid 0 \in A x\}$  of zeros of $A$ is nonempty. Recall that for any $\gamma > 0$, 
the resolvent of order $\gamma$ of $A$ is defined by $J_{\gamma A} = (\operatorname{Id} + \gamma A)^{-1}$, 
where $\operatorname{Id}$ is the identity mapping of $X$.  As $A$ is maximally monotone,  $J_{\gamma A}: X\to X$ is 
a single-valued firmly nonexpansive (in particular, nonexpansive) mapping such that  
$\operatorname{Fix}(J_{\gamma A})=\operatorname{zer}(A)$ for 
every $\gamma > 0$. We refer to \cite{BauCom17} for details. 

The \emph{Halpern-type proximal point algorithm} was introduced by Kamimura and Takahashi \cite{KamTak00} as a combination 
between the Halpern iteration and the proximal point algorithm:
\begin{align}
x_0=x, \qquad & x_{n + 1} = (1 - \alpha_n)u + \alpha_nJ_{\gamma_n A}x_n, \label{eq:hppa}
\end{align}
where $x,u\in C$, $(\alpha_n)$ is a sequence in $[0, 1]$, and  $(\gamma_n)$ a sequence in $(0, \infty)$.  A generalization of the iteration 
\eqref{eq:hppa} to $m$-accretive operators in Banach spaces was studied by Dominguez-Benavides, L\'opez-Acedo, and Xu \cite{DomLopXu03}.

If $\alpha_n=1$, then $(x_n)$ becomes the proximal point algorithm, while if $\gamma_n=\gamma$ is a constant sequence, then
$(x_n)$ is the Halpern iteration with $T=J_{\gamma A}$. 

By \cite[Proposition~3.21]{LeuNicSip18}, the family $(J_{\gamma_nA})$ is jointly 
firmly nonexpansive with respect to $(\gamma_n)$,  hence \eqref{main-ineq-HPPA} holds with $T_n=J_{\gamma_nA}$. Furthermore, 
$\bigcap_{n \in \N} \operatorname{Fix}(J_{\gamma_n A})=\operatorname{zer}(A)\ne \emptyset$. 
    
It follows that we can apply Proposition~\ref{linear-rates-HAPP} with $T_n=J_{\gamma_nA}$ to 
get, for $\alpha_n = 1 - \frac{2}{n + 2}$ and $\gamma_n=\frac{n+2}{n+1}$, that $(x_n)$ is ($(T_n)$, $T_m$)-asymptotically 
regular with linear rates given by \eqref{linear-rates-HAPPA}. This is a significant
improvement of the polynomial rates obtained by the second author and Pinto 
in \cite[Proposition~9]{LeuPin21} for another choice of the parameter sequences. 
   
\subsection{Minimizers of proper convex lsc functions}

Let $X$ be a complete CAT(0) space and $f : X \to (-\infty, +\infty]$ be a proper convex 
lower semicontinuous (lsc) function  that has minimizers, that is the set 
$\operatorname{Argmin}(f)=\{x\in X \mid f(x) \leq f(y) \text{~for all~}y\in X\}$ of minimizers of $f$ is nonempty. 
For any $\gamma > 0$,  the resolvent $J_{\gamma f}$ of $f$ of order $\gamma$ 
was defined in a nonlinear setting  by Jost \cite{Jos95}:
\begin{equation*}
    J_{\gamma f}: X\to X, \quad   J_{\gamma f}(x) = \operatorname{argmin}_{y \in X} \left(f(y) + \frac{1}{2\gamma} d^2(x, y)\right).
\end{equation*}

Jost proved in \cite{Jos95} that for all $\gamma > 0$,  $J_{\gamma f}$ is nonexpansive and 
$\operatorname{Fix}(J_{\gamma f}) =\operatorname{Argmin}(f)$. In fact, $J_{\gamma f}$ is even firmly nonexpansive, as pointed 
out in \cite[Proposition~3.3]{AriLeuLop14}. 

We have that $\bigcap_{n \in \N} \operatorname{Fix}(J_{\gamma_n f})=\operatorname{Argmin}(f)\ne \emptyset$ and that 
\eqref{main-ineq-HPPA} holds with $T_n=J_{\gamma_nf}$,  since the family $(J_{\gamma_nf})$ is jointly 
firmly nonexpansive with respect to $(\gamma_n)$, as  proved in \cite[Proposition~3.17]{LeuNicSip18}. 
    
It follows that we can apply Proposition~\ref{linear-rates-HAPP}  with $T_n=J_{\gamma_nf}$ to 
get linear rates  for $\alpha_n = 1 - \frac{2}{n + 2}$ and $\gamma_n=\frac{n+2}{n+1}$. 

\mbox{}

\section*{Acknowledgements}

The authors thank Andrei Sipo\c{s} for providing comments and suggestions that improved the final version of the paper.

\end{document}